\documentclass[12pt, a4paper,reqno]{amsart}

\usepackage[margin=2.5cm]{geometry}

\usepackage{amssymb,latexsym,amsmath,amsfonts,amsthm,amsxtra,amsrefs}
\usepackage[ocgcolorlinks,hyperfootnotes=false,colorlinks=true,citecolor=blue,linkcolor=blue,urlcolor=blue]{hyperref}
\usepackage{enumitem}

\newcommand{\C}{\mathbb C}
\newcommand{\D}{\mathbb D}
\newcommand{\E}{\mathbb E}
\newcommand{\G}{\mathbb G}
\renewcommand{\P}{\mathbb P}
\newcommand{\R}{\mathbb R}
\newcommand{\Z}{\mathbb Z}

\newcommand{\abs}[1]{\lvert #1\rvert}

\newcommand{\n}{\lVert}

\newtheorem{defn}{Definition}[section] % This Section declares
\newtheorem{prop}[defn]{Proposition}
\newtheorem{cor}[defn]{Corollary}
\newtheorem{lemma}[defn]{Lemma}
\newtheorem{thm}[defn]{Theorem}
\newtheorem{rem}[defn]{Remark}
\newtheorem*{MainThm}{Theorem \ref{main result}}
\newtheorem*{thmnonum1}{Theorem \ref{tetra main}}
\newtheorem*{thmnonum2}{Theorem \ref{penta main}}
\newtheorem*{thmnonum3}{Theorem \ref{symm main}}

\begin{document}

\title{The Friedrichs Operator and Circular Domains}

\author[Sivaguru Ravisankar]{Sivaguru Ravisankar}
\address[Sivaguru Ravisankar]{Tata Institute of Fundamental Research Centre for Applicable Mathematics, Bengaluru 560065, India.} \email{sivaguru@tifrbng.res.in}

\author[Samriddho Roy]{Samriddho Roy}
\address[Samriddho Roy]{Theoretical Statistics and Mathematics Unit,
Indian Statistical Institute, Delhi centre, 
New Delhi 110016, India} \email{samriddhoroy@gmail.com}

\keywords{Friedrichs Operator, Bergman Space, Circular Domain, Tetrablock, Pentablock, Symmetrized Polydisc}

\subjclass[2020]{32A36 32Q02 (Primary), 93D21 (Secondary)}

\begin{abstract}
The Friedrichs operator of a domain (in $\C^n$) is closely related to its Bergman projection and encodes crucial information (geometric, quadrature, potential theoretic etc.) about the domain.
We show that the Friedrichs operator of a domain has rank one if the domain can be covered by a circular domain via a proper holomorphic map of finite multiplicity whose Jacobian is a homogeneous polynomial.
As an application, we show that the Friedrichs operator is of rank one on the tetrablock, pentablock, and the symmetrized polydisc -- domains of significance in the study of $\mu$-synthesis in control theory.
\end{abstract}
%\date{\today}
\date{July 28, 2023}
\maketitle

\section{Introduction}

Let $D\subset\C^n$, $n\ge 1$, be a domain.
The Bergman space of $D$, denoted by $A^2(D)$, is the set of holomorphic functions in $L^2(D)$.
The Friedrichs operator of $D$ is
\begin{equation}
    F\colon A^2(D) \to A^2(D) \text{ defined by } F(g) = B(\bar{g})
\end{equation}
where the Bergman projection $B$ is the orthogonal projection from $L^2(D)$ onto $A^2(D)$.
The Bergman projection and the Friedrichs operator encode crucial information (geometric, quadrature, potential theoretic etc.) about the domain.
For instance, the Friedrichs operator having finite rank translates to the domain satisfying a quadrature identity which, in turn, has close connections to the boundary geometry of the domain.
For more on this, see Friedrichs \cite{Friedrichs} and Shapiro \cite{Shapiro}.

A domain $D$ is said to be {\it circular} if $e^{i\theta}z\in D$ whenever $z\in D$ and $\theta\in\R$.
We first show that the Friedrichs operator of a circular domain containing $0$ has rank one -- see Proposition~\ref{prop:FriedCircRank1}.
This is an easy consequence of a projective representation of circular domains due to Azukawa \cite{Azukawa}.
Our main result in this article is the following generalization which we prove in Section~\ref{Trybula}.
\begin{MainThm}
Let $D_1$ and $D_2$ be bounded domains in $\C^n$, $n\ge 1$, and $\phi\colon D_1\to D_2$ be a proper holomorphic map of finite multiplicity.
If $D_1$ is circular, $0\in D_1$, and $J\phi$ is a homogeneous polynomial, then the Friedrichs operator of $D_2$ is of rank one.
\end{MainThm}
The approach of using a covering domain to study the Bergman projection and related objects has its origins in the work of Misra, Roy, and Zhang \cite{GM13} and has been refined further by Trybu{\l}a \cite{Try13}.
We build on results of Azukawa \cite{Azukawa}, to show that the Friedrichs operator associated to a certain weighted Bergman space on  circular domain has rank one.
We then prove generalizations of some results of Trybu{\l}a which allow us to conclude our main result.

Using the above result we show the Friedrichs operator has rank one on many domains that are of significance in the $\mu$-synthesis problem from control theory.
For a brief description of $\mu$-synthesis see \S\ref{subsec:muSynth}.
In some cases, the $\mu$-synthesis problem reduces to an interpolation problem from $\D$ to certain special domains.
Some examples of these special domains are the tetrablock, pentablock, and symmetrized polydisc.
Chen, Krantz, and Yuan \cite{Krantz-Chen} show that the Friedrichs operator of the symmetrized bidisc $\G_2$ is of rank one.
Their proof relies crucially on covering $\G_2$ by $\D^2$, a Reinhardt domain.
By allowing for covering domains to be circular (possibly non-Reinhardt), we prove the following in Section~\ref{sec:FriedrichsOnMuDomains}.
\begin{thmnonum1}
The Friedrichs operator of the Tetrablock $\E\subset\C^3$ has rank one.
\end{thmnonum1}
\begin{thmnonum2}
The Friedrichs operator of the Pentablock $\mathcal{P}\subset\C^3$ has rank one.
\end{thmnonum2}
\begin{thmnonum3}
The Friedrichs operator of the Symmetrized polydisc $\G_n\subset \C^n$, $n\ge 2$, has rank one.
\end{thmnonum3}
We now begin with a quick overview of Bergman spaces, circular domains, and Bergman spaces of circular domains.

%%%%%%%%%%%%%%%%%%%%%%%%%%%%%%%%%%%%%%%%%%%%%%%%%%%%%%%%%%

\section{Bergman Spaces and Circular Domains}\label{Bergman theory}

Let $D \subset \mathbb C^n$ be a bounded domain, $dV$ denote the Lebesgue measure in $\mathbb C^n$, and $L^2(D)$ denote the Hilbert space of square-integrable functions with the inner product
\begin{equation}
\langle f,g \rangle = \int_{D}^{} f(z) \overline{g(z)} dV(z).
\end{equation}
The subspace of $L^2(D)$ consisting of holomorphic functions is called as the {\it Bergman space} of $D$ and we denote it by $A^2(D)$. 

The Bergman space $A^2(D)$ is a closed subspace of $L^2(D)$ and this induces an orthogonal projection from $L^2(D)$ onto $A^2(D)$ known as the {\it Bergman projection} of $D$, denoted by $B_{D}$.
The Freidrichs operator of $D$ is defined as $F_{D}\colon A^2(D) \to A^2(D)$, $F_{D}(g) = B_{D}(\bar g)$.

The Bergman projection is an integral operator given by
\begin{equation}
B_{D}(f)(z) =  \int_{D} K_D (z,w)f(w) dV(w),    
\end{equation}
where the integral kernel $K_D(z,w)$ is known as the {\it Bergman kernel} of $D$.
We will drop the subscripts when the domain under consideration is clear from context.
The Bergman kernel $K\colon D \times D \to \C$ is a reproducing kernel for the Bergman space and satisfies the following properties:
\begin{enumerate}\itemsep1mm
    \item $k_w:=K(\cdot , w) \in A^2(D)$ for all $w \in D$,
    \item $\langle f, k_w \rangle = f(w)$ for all $f \in A^2(D)$ and $w \in D$, and
    \item if $\{e_n\}$ is an orthonormal basis for $A^2(D)$, then
\begin{equation}
K(z,w) = \sum\limits e_n(z) \overline{e_n(w)}.
\end{equation}
\end{enumerate}

A domain $D\subset\C^n$ is said to be {\it circular} if $e^{i\theta}z\in D$ for every $z\in D$ and $\theta\in\R$.
Circular domains admit a characterization in terms of projective coordinates that is useful.
Let $D\subset\C^n$ be a circular domain.
Define
\begin{equation}
    V=\{(\zeta,r)\in\C\P^{n-1}\times\R_{\ge 0}\colon r\psi(\zeta)\in D\}
\end{equation}
where $\psi\colon\C\P^{n-1}\to S^{2n-1}$ is such that $\pi\circ\psi = \operatorname{Id}_{\C\P^{n-1}}$ and $\pi\colon\C^n\setminus\{0\}\to\C\P^{n-1}$ is the canonical projection.
The set $V$ is independent of the choice of $\psi$ and is called the {\it representative domain} for $D$.
The domain $D$ can be recovered from $V$ as follows:
\begin{equation}
    D=\{re^{i\theta}\psi(\zeta) \colon (\zeta,r)\in V,\theta\in\R\}.
\end{equation}

The Bergman space of a circular domain admits a decomposition in terms of homogeneous polynomials.
A holomorphic function $f$ on $D$ is said to be {\it $k$-homogeneous}, for $k\in\Z$, if $f(\lambda z)=\lambda^k f(z)$ for all $z\in D$ and $\lambda\in\C$ with $\abs{\lambda}\in I(z)$ where $I(z)$ is the connected component of the set $\{r\in\R\colon r\ne 0, rz\in D\}$ that contains $1$.
Let $H_k(D)$ denote the set of holomorphic functions on $D $ that are $k$-homogeneous.
If $f$ is holomorphic in $D$, then $f=\sum_{k\in\Z}f_k$, with $f_k\in H_k$, and the series converges uniformly on compact subsets of $D$ (see \cite{Azukawa}*{Lemma 1.3}).

The following lemma is an expanded version of \cite{Azukawa}*{Lemma 1.2}.
To state this lemma we need to express integrals on $D$ as integrals on $V$ using the Fubini-Study metric in $\C\P^{n-1}$.
Let $U=\{\pi(z_1,\ldots,z_n)\colon z_n\ne 0\}$.
Let $u_j\colon U\to \C$, for $1\le j\le n-1$, be defined by $u_j(\pi(z))=z_j/z_n$, and $u=(u_1,\ldots,u_{n-1})$.
Let $v$ denote the volume element on $\C\P^{n-1}$ associated to the Fubini-Study metric.
Then,
\begin{equation}\label{eqn:FSandEucVol}
    v\vert_U = \left(1+\abs{u}^2\right)^{-n}u^{\ast} dV_{n-1}
\end{equation}
where $dV_{n-1}$ is the volume element in $\C^{n-1}$.
Let $\alpha\colon U\to S^{2n-1}$ be given by
\begin{equation}
    \alpha = \left(1+\abs{u}^2\right)^{-1/2}(u_1,\ldots,u_{n-1},1).
\end{equation}

\begin{lemma}\label{lemma:circular-Azukawa}
	Let $D\subset\C^n$ be a circular domain with representative domain $V\subset\C\P^{n-1}\times\R_{\ge 0}$.
	\begin{enumerate}[label=(\alph*)]
	    \item For $f,g\in A^2(D)$,
	    \begin{equation}
	        \langle f,g\rangle = \int\limits_0^{2\pi}\int\limits_{(\zeta,r)\in V,\zeta\in U}\,f\big(r\alpha(\zeta)e^{i\theta}\big)\overline{g\big(r\alpha(\zeta)e^{i\theta}\big)}r^{2n-1}v(\zeta)\wedge dr\wedge d\theta.
	    \end{equation}
	   \item For $f \in H_k$, $g \in H_\ell$, and $k \neq \ell$, we have $\langle f,g \rangle =0$.
	   \item For $f\in A^2(D)$, let $f=\sum_{k\in\Z} f_k$, with $f_k\in H_k$, be the homogeneous expansion.
	   Then, $f_k\in A^2(D)$ for each $k$.
	   \item $A^2(D) = \bigoplus\limits_{k\in\Z}^{} H_k\cap A^2(D)$.
\end{enumerate}
\end{lemma}
As an easy consequence of the above lemma, we get that the Friedrichs operator of a circular domain containing the origin is of rank one.
\begin{prop}\label{prop:FriedCircRank1}
	Let $D\subset\C^n$ be a circular domain with $0\in D$.
	Then, the Friedrichs operator of $D$ is of rank one.
\end{prop}
\begin{proof}
	For $f \in A^2(D)$, $f=\sum_{k \ge 0} f_k$, where $f_k \in A^2(D)$ is $k$-homogeneous.
	Note that $k\ge 0$ in this expansion since $0\in D$.
	Let $F_D$ be the Friedrichs operator of $D$. Then, 
	\begin{equation}
	 F_D(f) = B_D(\bar f) = B_D\big(\sum\limits_{k} \bar{f}_k\big) = \sum\limits_{k} B_D(\bar{f}_k).
	\end{equation}
	For $k>0$, Lemma \ref{lemma:circular-Azukawa} gives us that $\langle \bar{f}_k,g_\ell\rangle = 0$ for all $g_\ell\in H_\ell$ and $\ell\ge 0$.
	So, $\langle \bar{f}_k,g\rangle = 0$ for all $g\in A^2(D)$.
	Since $f_0$ is a constant, $F_D(f)=B_D(\bar{f}_0)$ and the conclusion follows.
\end{proof}
If the circular domain $D$ does not contain the origin, the Friedrichs operator can have arbitrary finite rank or infinite rank.
To realize the former possibility consider the {\it fat} Hartogs triangle $\{\abs{z}^\gamma < \abs{w} < 1\}$,  $\gamma >0$, and for the latter consider a product of annuli centred at the origin.
For more on these examples and related ideas see Ravisankar and Zeytuncu \cite{RZ}.

%%%%%%%%%%%%%%%%%%%%%%%%%%%%%%%%%%%%%%%%%%%%%%%%%%%%%%%%%

\section{Proper holomorphic mappings and the Friedrichs operator}\label{Trybula}
In this section we present the relationship between the Bergman projection of a domain and its image under a proper holomorphic map of finite multiplicity.
This, in turn, leads to a relationship between their Friedrichs operators.

Let $D_1$ and $D_2$ be two bounded domains in $\mathbb C^n$, $n\ge 1$, and $\phi: D_1 \rightarrow D_2$ be a proper holomorphic mapping with multiplicity $m$.
Trybula \cite{Try13} has shown that there is a closed subspace of $A^2(D_1)$ which is unitarily isomorphic to $A^2(D_2)$ (see also \cite{GM13}).
Let $J\phi$ denote the complex Jacobian of  $\phi$ and $\nu(z) = |J\phi(z)|^2$.
Let the weighted Bergman space $A^2(D_1, \nu)$ be the set of holomorphic functions in $L^2(D_1, \nu)$ equipped with the inner product $\langle f,g \rangle_\nu = \int_{D_1}^{} f(z) \overline{g(z)} \nu(z) dV(z)$.
We adopt the same strategy as Trybula \cite{Try13} to generalize their results to our setting.
We show that there is a closed subspace of $A^2(D_1, \nu)$ which is unitarily isomorphic to $A^2(D_2)$.
For $f \in A^2(D_2)$, $f \circ \phi$ is well defined and holomorphic on $D_1$ and, by change of variables,
 \begin{equation}\label{cov}
 m \int\limits_{D_2}^{} f dv = \int\limits_{D_1}^{} (f \circ \phi) |J\phi|^2 dv.
 \end{equation}	
Thus $f \circ \phi  \in A^2(D_1, \nu)$.
Define $\Gamma_{\nu} : A^2(D_2) \longrightarrow A^2(D_1, \nu)$ by $\Gamma_{\nu}(f) = \dfrac{1}{\sqrt{m}} (f \circ \phi)$. Clearly $\Gamma_{\nu}$ is an isometric embedding. Therefore $\Gamma_{\nu} A^2(D_2)$ is a closed subspace of $A^2(D_1, \nu)$ that is isometrically isomorphic to $A^2(D_2)$.

Note that $\Gamma_{\nu}$ is unitary when understood as an operator from $A^2(D_2)$ onto $\Gamma_{\nu} A^2(D_2)$.
The adjoint operator $\Gamma_{\nu}^\ast$ can be described as follows.
Let $g \in \Gamma_{\nu} A^2(D_2)$.
Then, $g(z) = g(w)$ whenever $\phi(z) = \phi(w)$ and $z,w\in D_1$.
So, we can define $\widetilde g$ on $D_2$ by $\widetilde{g}(\phi(z)) = g(z)$.
Then, $\widetilde g$ is well defined and holomorphic on $D_2$.
It is easy to verify, using \eqref{cov}, that $\widetilde g \in A^2(D_2)$. Hence,  \begin{equation}\label{adjoint}
 \Gamma_{\nu}^\ast(g) = \sqrt{m}\,\widetilde{g}, \; \text{ for } g \in \Gamma_{\nu} A^2(D_2).
\end{equation}
 
 \begin{rem}\label{rem:img Gamma_nu}
 For $g \in \Gamma_{\nu} A^2(D_2)$, $\widetilde{g} \in A^2(D_2)$ and $\widetilde{g} \circ \phi = g$.
 Therefore $\Gamma_{\nu}^\ast = \Gamma_{\nu}^{-1}$ (where $\Gamma_\nu$ is considered a map onto its range) and
  \begin{equation}
 \Gamma_{\nu} A^2(D_2) = \{ g \in A^2(D_1, \nu) \colon g(z) = g(w) \text{ whenever } \phi(z) = \phi(w) \text{ and } z,w \in D_1 \}.
 \end{equation}
 \end{rem}
 
The following two lemmas express the Bergman kernel of $D_2$ in terms of the weighted Bergman kernel of $D_1$.
 \begin{lemma}\label{Try-projection}
 	The orthogonal projection $P_{\nu}$ of $A^2(D_1, \nu)$ onto $\Gamma_{\nu} A^2(D_2)$ is given by 
 	\begin{equation}\label{projection formula}
 	P_{\nu}g = \dfrac{1}{m} \sum\limits_{j=1}^{m} (g \circ \phi^j \circ \phi), \; \text{ for } g \in A^2(D_1, \nu),
 	\end{equation}
 	where $\{\phi^j\}_{1}^{m}$ are the local inverses of $\phi$.
 \end{lemma}
 \begin{proof}
 	For $g \in A^2(D_1, \nu)$, let $Qg$ denote the
 	right hand side of \eqref{projection formula}.
    Note that $Qg \in A^2(D_1, \nu)$ since $Qg$ is holomorphic and
 	\begin{align}
 	\n Qg \n_{(D_1, \nu)} &= \dfrac{1}{m^2} \int\limits_{D_1}^{} \abs{\sum\limits_{j=1}^{m}(g \circ \phi^j \circ \phi)}^2 |J\phi|^2 dv\\
 	& \leq \dfrac{1}{m} \int\limits_{D_1}^{} \sum\limits_{j=1}^{m} |(g \circ \phi^j \circ \phi)|^2 |J\phi|^2 dv = \n g \n_{(D_1, \nu)}.
 	\end{align}
 	Using $\phi \circ \phi^j \circ \phi = \phi$, it is easy to verify that $Q^2 = Q$ and $Q \circ \Gamma_{\nu} = \Gamma_{\nu}$.
 	Since $Qg(z)=Qg(w)$ whenever $\phi(z)=\phi(w)$ and $z,w\in D_1$, we have $\widetilde{Qg}\in A^2(D_2)$ with $\widetilde{Qg}\big(\phi(z)\big) = Qg(z)$.
 	Therefore the range of $Q$ coincides with the range of $\Gamma_{\nu}$.
 	Since $Q$ is a projection and $\n Q\n = 1$, we conclude that $Q$ is the orthogonal projection onto $\Gamma_\nu A^2(D_2)$.
 \end{proof}

 \begin{lemma}\label{Try-kernel}
 	Let $K_{D_1}^{\nu}$ and $K_{D_2}$ be the Bergman kernels associated to the Bergman projections onto $A^2(D_1,\nu)$ and $A^2(D_2)$ respectively.
 	Then 
 	\begin{equation}\label{kernel interplay}
 	K_{D_2}(\phi(z), \phi(w)) = \sum\limits_{j=1}^{m}   K_{D_1}^{\nu}(\phi^j \circ \phi (z), w)
 	\end{equation}
 	where $\{\phi^j\}_{1}^{m}$ are the local inverses of $\phi$.
 \end{lemma}
 \begin{proof}
 	For $f\in A^2(D_2)$ and $w \in D_1$, $K_{D_1}^{\nu}(\cdot,w)\in A^2(D_1,\nu)$ and hence
 	\begin{equation}
 	\langle \Gamma_{\nu}f, (I-P_{\nu})K_{D_1}^{\nu}(\cdot,w) \rangle_{(D_1,\nu)} = 0
 	\end{equation}
 	where $P_{\nu}$ is the orthogonal projection of $A^2(D_1, \nu)$ onto $\Gamma_{\nu} A^2(D_2)$.
   	By the reproducing property of $K_{D_2}$ and $\Gamma_{\nu}$ being an isometry, we have
 	\begin{align}
 	\langle f, \Gamma_{\nu}^{\ast}P_{\nu}K_{D_1}^{\nu}(\cdot,w) \rangle_{D_2} &= \langle \Gamma_{\nu}f, P_{\nu}K_{D_1}^{\nu}(\cdot,w) \rangle_{(D_1, \nu)} 
 	= \langle \Gamma_{\nu}f, K_{D_1}^{\nu}(\cdot,w) \rangle_{(D_1, \nu)}\\
 	&= \Gamma_{\nu}f(w)
 	= \dfrac{1}{\sqrt{m}} f(\phi(w))
 	 = \dfrac{1}{\sqrt{m}} \langle f, K_{D_2}(\cdot, \phi(w)) \rangle_{D_2}.
 	\end{align}
 	 So,
	$K_{D_2}(\cdot, \phi(w)) = \sqrt{m}\,\Gamma_{\nu}^{\ast} \big(P_{\nu}K_{D_1}^{\nu}(\cdot ,w)\big)$.
 	Now, by \eqref{adjoint} and Lemma~\ref{Try-projection},
 	\begin{equation}
 	   K_{D_2}(\phi(z), \phi(w))
 	   =mP_{\nu}K_{D_1}^{\nu}(z ,w)
 	   =\sum\limits_{j=1}^{m} K_{D_1}^{\nu}(\phi^j \circ \phi(z) ,w).
 	\end{equation}
 \end{proof}

\begin{rem}
The relation between the Bergman kernels $K_{D_1}$ and $ K_{D_1}^{\nu}$ is given by 
\begin{equation}\label{kernel interplay 1}
K_{D_1}(z,w) = J\phi(z)  K_{D_1}^{\nu}(z,w) \overline{J\phi(w)}.
\end{equation}
In fact, if $\{\varphi_n\}_{n=1}^{\infty}$ is an orthonormal basis for $A^2(D_1, \nu)$, then $\{\varphi_n J\phi\}_{n=1}^{\infty}$ is an orthonormal basis for $A^2(D_1)$.
 Since $ \phi \circ \phi^j \circ \phi = \phi$, we have $J\phi(\phi^j \circ \phi(z))J\phi^j(\phi(z))J\phi(z) = J\phi(z)$.
 An alternate approach to deducing \eqref{kernel interplay} is to use \eqref{kernel interplay 1} along with Corollary 1 of \cite{Try13}. 
\end{rem} 

The following results are consequences of Lemma \ref{Try-kernel}.
\begin{lemma}\label{bergman projection}
	Let $B_{D_1}^{\nu}$ and $B_{D_2}$ be the Bergman projections associated to $A^2(D_1, \nu)$ and $A^2(D_2)$ respectively. Then,
	\begin{equation}
	B_{D_2}(g)(\zeta) = \dfrac{1}{m} \sum\limits_{j=1}^{m} B_{D_1}^{\nu}(g \circ \phi)(\phi^j(\zeta)), \; \text{ for }  g \in L^2(D_2) \text{ and } \zeta \in D_2,
	\end{equation}
	 where $\{\phi^j\}_{1}^{m}$ are the local inverses of $\phi$.
\end{lemma}
\begin{proof}
	 Let $g \in L^2(D_2)$.
	 Then, $g \circ \phi \in L^2(D_1,\nu)$, 
	 \begin{align}
	 B_{D_1}^{\nu}(g \circ \phi)&=  \int_{D_1} K_{D_1}^{\nu} (\cdot,w)(g \circ \phi)(w) |J\phi(w)|^2 dV(w), \text{ and }\\
	 B_{D_2}(g) &=  \int_{D_2} K_{D_2} (\cdot,\eta)g(\eta) dV(\eta)\\
	 &= \dfrac{1}{m} \int_{D_1} K_{D_2} (\cdot,\phi(w))g(\phi(w)) |J\phi(w)|^2 dV(w).
	 \end{align}
	For $\zeta\in D_2$, there exists $z\in D_1$ such that $\zeta=\phi(z)$.
	Now, by Lemma~\ref{Try-kernel},
	\begin{align}
	B_{D_2}(g)(\zeta) =B_{D_2}(g)(\phi(z))&=  \dfrac{1}{m} \int_{D_1} K_{D_2} (\phi(z),\phi(w)) (g\circ \phi)(w) |J\phi(w)|^2 dV(w) \\
	&= \dfrac{1}{m} \int_{D_1} \sum\limits_{j=1}^{m}   K_{D_1}^{\nu}(\phi^j \circ \phi (z), w) (g\circ \phi)(w) |J\phi(w)|^2 dV(w)\\
	&= \dfrac{1}{m} \sum\limits_{j=1}^{m} B_{D_1}^{\nu}(g \circ \phi)(\phi^j(\zeta)).
	\end{align}
\end{proof}

\begin{cor}\label{cor:FriedRank1}
    Let $T\colon \Gamma_{\nu} A^2(D_2) \to A^2(D_1,\nu)$ be defined by $T(f) = B_{D_1}^\nu(\bar{f})$.
    If $T$ is of rank one, then  so is the Friedrichs operator on $D_2$.
\end{cor}
\begin{proof}
    Since $T(\lambda)=\bar{\lambda}$ for any $\lambda\in\C$, the range of $T$ is the set of (complex) constant functions.
    
	For $g \in A^2(D_2)$, we have $g \circ \phi \in \Gamma_{\nu} A^2(D_2)$ and $T(g \circ \phi)$ is a constant, say $a_0$.
	Now, by Lemma \ref{bergman projection},
	\begin{equation}
	    B_{D_2}(\bar g)(w) =  \dfrac{1}{m} \sum\limits_{j=1}^{m} B_{D_1}^{\nu}(\overline{g \circ \phi})(\phi^j(w)) = a_0,
	\end{equation}
	for $w \in D_2$.
\end{proof}
We now prove the main result of this article.
\begin{thm}\label{main result}
Let $D_1$ and $D_2$ be bounded domains in $\C^n$, $n\ge 1$, and $\phi\colon D_1\to D_2$ be a proper holomorphic map of finite multiplicity.
If $D_1$ is circular, $0\in D_1$, and $J\phi$ is a homogeneous polynomial, then the Friedrichs operator of $D_2$ is of rank one.
\end{thm}
\begin{proof} 
With the setup as in the beginning of this section, $\Gamma_{\nu} A^2(D_2)$ is a closed subspace of $A^2(D_1, \nu)$ that is isometrically isomorphic to $ A^2(D_2)$.

Since $J\phi$ is a homogeneous polynomial and $\nu=\abs{J\phi}^2$, a version of Lemma~\ref{lemma:circular-Azukawa} holds for $A^2(D_1,\nu)$.
Note that $0\in D_1$ and hence any homogeneous polynomial in $D_1$ is $k$-homogeneous for some $k\ge 0$.
We use Lemma~\ref{lemma:circular-Azukawa} to get that 
\begin{align}
\langle f_k, f_{\ell}\rangle_{D_1,\nu}
& = \int_{D_1} f_k(x,y,z) \overline{f_{\ell}(x,y,z)} \nu(z) dV(x,y,z) \\
&=  \int_{D_1} \left((J\phi)f_k(x,y,z)\right) \left(\overline{(J\phi)f_{\ell}(x,y,z)}\right) dV(x,y,z) =0,
\end{align}
for $f_k \in H_k$, $f_{\ell} \in H_{\ell}$, and $k \neq \ell$.
Hence,
\begin{equation}
    A^2(D_1,\nu) = \bigoplus\limits_{k\geq 0}^{} H_k\cap A^2(D_1,\nu).
\end{equation}
Lemma~\ref{lemma:circular-Azukawa}(a) also gives us that 
$\langle \bar{f}_k,f_{\ell}\rangle_{D_1,\nu} = 0$
for $f_k\in H_k, f_{\ell}\in H_{\ell}$ unless
$k=\ell=0$.

Let $T\colon \Gamma_{\nu} A^2(D_2) \to A^2(D_1,\nu)$ by $T(f) = B_{D_1}^\nu(\bar{f})$.
Write $f\in\Gamma_{\nu} A^2(D_2)$ as $f=\sum_{k\ge 0} f_k$, where $f_k\in H_k$, to get
\begin{equation}
T(f) = B_{D_1}^\nu\big(\bar{f}\big)
=B_{D_1}^{\nu}\Big(\sum\limits_{k\ge 0}^{} \bar{f}_k\Big)
=\sum\limits_{k\ge 0}^{}B_{D_1}^{\nu}\big(\bar{f}_k\big)
= B_{D_1}^\nu\big(\bar{f}_0\big) \text{ (a constant)}.
\end{equation}
Thus $T$ is of rank one.
Now, by Corollary~\ref{cor:FriedRank1}, the Friedrichs operator of $D_2$ is of rank one.
\end{proof}

%%%%%%%%%%%%%%%%%%%%%%%%%%%%%%%%%%%%%%%%%%%%%%%%%%%%%%%%%%%%

\section{Friedrichs Operator of domains related to \texorpdfstring{$\mu$}{Mu}-synthesis}\label{sec:FriedrichsOnMuDomains}

In this section we consider three domains related to $\mu$-synthesis -- the tetrablock, the pentablock and the symmetrized polydisc.
For each of these domains, we first recall some important characterizations that define them.
We then proceed to show, using Theorem~\ref{main result}, that their Friedrichs operators are of rank one.
We now begin with a brief description of the $\mu$-synthesis problem.

\subsection{\texorpdfstring{$\mu$}{Mu}-Synthesis}\label{subsec:muSynth}

The $\mu$-synthesis problem plays an important role in modelling structured uncertainties in control engineering.
Here, $\mu$ is a cost function on matrices that denotes the structured singular value of a matrix relative to a subspace of linear transformations (see \cites{BFT90,doyle}).
Let $M$ be a linear subspace of the complex $n\times m$ matrices $\mathbb C^{n \times m}$.
The {\em structured singular value} of an $B\in\C^{m \times n}$, denoted by $\mu_M(B)$, is defined as
\begin{equation}
\mu_M(B) = \frac{1}{\inf \left\{ \n X \n \colon X \in M, I-BX  \mbox{ is singular}\right\}},    
\end{equation}
where $\n X\n$ denotes the operator norm of the matrix $X$.
We set $\mu_M(B)=0$ whenever $I-BX$ is non-singular for all $X \in M$.

Let $\lambda_1, \dots, \lambda_k$ be distinct points in the unit disc $\mathbb D \subset \mathbb{C}$, and $B_1, \dots, B_k\in\C^{m \times n}$.
The $\mu$-synthesis problem is to find an analytic function $f \colon \mathbb D \to \mathbb C^{m \times n}$ such that
\begin{equation}
    f(\lambda_j)=B_j, \text{ for } 1 \leq j \leq k, \quad\text{and}\quad
\mu_M\big(f(\lambda)\big)<1, \text{ for } \lambda \in \mathbb D.
\end{equation}

When $M$ is the set of scalar square matrices ($m=n$), $\mu_M$ coincides with the spectral radius.
In that case, the $\mu$-synthesis problem reduces to the spectral Nevanlinna-Pick interpolation problem. 
Agler and Young \cite{AY99} showed that this problem, in $m=n=2$, reduces to an interpolation problem from $\D$ to the symmetrized bidisc $\G_2$.
A similar phenomenon holds for certain other subspaces of $\C^{2\times 2}$ reducing the $\mu$-synthesis problem to an interpolation problem into the domains tetrablock $\E$ and pentablock $\mathcal{P}$.

\subsection{Tetrablock}
The tetrablock is defined as
	\begin{equation}\label{tetra}
	\mathbb E = \{ (x_1,x_2,x_3)\in\mathbb C^3\,:\, 1-zx_1-wx_2+zwx_3\neq 0 \,, z,w \in\overline{\mathbb D} \}
	\end{equation}
and can be characterized as follows.
\begin{thm}[{\cite{awy}*{Theorem 2.4}}] \label{char tetra}
	For $x \in \mathbb C^3$, the following are equivalent.
	\begin{enumerate}
		\item $x \in \E$.
		\item $|x_1 - \bar x_2 x_3| + |x_1 x_2 - x_3| < 1-|x_2|^2$.
		\item $|x_2 - \bar x_1 x_3| + |x_1 x_2 - x_3| < 1-|x_1|^2$.
		\item $|x_1 - \bar x_2x_3| + |x_2 - \bar x_1 x_3| <1 - |x_3|^2$.
		\item There exists a symmetric $2 \times 2$ matrix $A=[a_{ij}]$ such that $||A|| <1$  and $x=(a_{11}, a_{22}, \det A)$.
		\item $|x_3| < 1$ and there exist $\beta_1, \beta_2 \in \mathbb C$ such that
		$|\beta_1| + |\beta_2| < 1$  and
		\[
		x_1=\beta_1 + \bar\beta_2 x_3, \quad x_2 = \beta_2+\bar\beta_1 x_3.
		\]
	\end{enumerate}
\end{thm}

The tetrablock is useful in the study of the $\mu$-synthesis problem where the structure is given by the diagonal matrices $M$ in  $\C^{2\times 2}$.
By \cite{awy}*{Theorem 9.1},
\begin{equation}
    \mathbb{E} = \{ (a_{11}, a_{22}, \det A) \colon A=[a_{ij}] \in \mathbb C^{2\times 2}, \mu_M(A) < 1\}.
\end{equation}
The interpolation problem for the unbounded $4$-dimensional domain $\Sigma = \{ A \in \mathbb C^{2\times 2}: \mu_M(A) < 1\}$ is equivalent to the interpolation problem for the bounded $3$-dimensional domain $\mathbb E$ (see \cite[Theorem 9.3]{awy}). 

The tetrablock is not a circular domain (see \cite{awy}*{Theorem 2.12}) and hence also not Reinhardt.
It is not a Hartogs domain either; to see this, note  
that $(1,1,1) \in \overline{\E}$, but none of the points $(i,1,1), (1,i,1) $ and $(1,1,i)$ are in $\overline{\E}$.
So, we use a proper holomorphic map from a circular domain to $\E$ to study the Friedrichs operator on $\E$.

\begin{thm}\label{tetra main}
	The Friedrichs operator of $\E$ is of rank one.
\end{thm}
\begin{proof}
Let $S$ be the unit ball, in the operator norm,  of symmetric complex $2\times 2$ matrices.
We view $S$ as a domain in $\C^3$:
\begin{equation}
    S= \{(x,y,z) \in \mathbb C^3 \colon |x|^2 + |y|^2 + 2|z|^2 < 1 + \abs{xy - z^2} \}.
\end{equation}
Note that $S$ is a circular domain containing the origin.
By Theorem \ref{char tetra}, we can write 
\begin{equation}
    \mathbb E = \{(x,y,xy- z^2) :  (x,y,z)\in S\}.
\end{equation}
Let $\Phi : S \rightarrow \mathbb E$ be defined by $\Phi (x,y,z) = (x,y,xy- z^2)$.
It is a proper holomorphic map of multiplicity two with Jacobian $J\Phi (x,y,z) = -2z$.
The conclusion follows from Theorem~\ref{main result} since the Jacobian is a homogeneous polynomial.
\end{proof}
\subsection{Pentablock}
The next domain we consider is the pentablock
	\begin{equation}\label{Penta}
	\mathcal{P}:=\{(a_{21},\operatorname{tr}(A), \operatorname{det}A)\in\C^3: A = [a_{ij}]\in\C^{2\times 2} , \|A\|<1\}.
	\end{equation}
Here are few alternate characterizations of $\mathcal{P}$.
\begin{thm}[{\cite[Theorem 1.1]{ALY15}}]\label{char penta}
	Let $(s,p)\in\mathbb{G}_2$ and $a\in\C$.
	The following are equivalent.
	\begin{enumerate}\itemsep10pt
		\item $(a,s,p)\in \mathcal{P}$.
		\item $|a|<\left|1- \dfrac{\frac{1}{2} s\bar\beta}{1+\sqrt{1-|\beta|^2}}\right|$, where $\beta=\dfrac{s-\bar sp}{1-|p|^2}$.
		\item $2|a| < |1-\bar\lambda_2\lambda_1|+ \sqrt{(1-|\lambda_1|^2)(1-|\lambda_2|^2)}$, where $\lambda_1,\lambda_2 \in\mathbb D$ and $(s,p)= (\lambda_1+\lambda_2,\lambda_1\lambda_2)$.  
	\end{enumerate}
\end{thm}
Similar to the tetrablock, the pentablock is also related to the $\mu$-synthesis problem where the structure is given by the upper triangular matrices $N$ in $\C^{2\times 2}$.
Additionally, by \cite[Theorem 5.2]{ALY15}, the pentablock can also be characterized as follows.
\begin{equation}
\mathcal{P}=\{(a_{21},\operatorname{tr}(A), \operatorname{det}A): A = [a_{ij}]\in\C^{2\times 2}, \mu_N(A) < 1\}.
\end{equation}

The pentablock is a Hartogs domain over the symmetrized bidisc $\mathbb G_2$.
But $\mathbb G_2$ itself is not a Hartogs domain (see \cite{Krantz-Chen}*{Proposition $6.3$}). It is easy to check that the pentablock is not a Reinhardt domain: $(0,2,1) \in \overline{\mathcal P}$, but $(0,2i, 1) \notin \overline{\mathcal P}$.

\begin{thm}\label{penta main}
	The Friedrichs operator of $\mathcal P$ is of rank one.
\end{thm}
\begin{proof}
Consider the circular domain
\begin{equation}
\mathcal L= \big\{(x,y,z) \in \mathbb C^3 \colon 2|z| < |1-x\bar y|+ \sqrt{(1-|x|^2)(1-|y|^2)} \big\}.
\end{equation}
Let $\Psi\colon\mathcal{L}\rightarrow\mathcal{P}$ be defined by $\Psi (x,y,z) = (z, x+y,xy)$.
It is easy to see that $\Psi$ is a proper holomorphic covering map with multiplicity two.
The Jacobian of $\Psi$, $J \Psi = x-y$, is a homogeneous polynomial.
Then, Theorem~\ref{main result} gives us the result.
\end{proof}

\subsection{Symmetrized polydisc}
The last domain we consider is the symmetrized polydisc $\mathbb G_n \subset\C^n$, a generalization of $\G_2$ to higher dimensions.
$\G_n$ is defined to be the image of the symmetrization map $\pi_n\colon\D^n\to\C^n$ defined by
	 \begin{equation}\label{pi_n}
	 \pi_n(z_1,\dots, z_n) = \left(\sum_{1\leq i\leq n} z_i,\sum_{1\leq
	 	i<j\leq n}z_iz_j,\dots, \prod_{i=1}^n z_i \right).
	 \end{equation}
However, the complex geometry and operator theoretic properties of $\G_n$, $n\ge 3$, is starkly different from those of $\G_2$ (see \cite{NPZ07}).

The symmetrized polydisc is associated with spectral interpolation and hence with the $\mu$-synthesis problem.
For $A\in\C^{n\times n}$,
the spectral radius $r(A)<1$ if and only if $\pi_n(\lambda_1,\dots, \lambda_n) \in \G_n$, where $\lambda_1, \dots , \lambda_n$ are
the eigenvalues of $A$ counted with multiplicities.
In fact, the interpolation problem into the spectral unit ball in $\mathbb C^{n \times n}$ is equivalent to the interpolation problem into $\mathbb G_n$ (see \cite{costara05}*{Theorem 2.1}).

We collect a few important characterizations of $\mathbb G_n$ in the following theorem.
More characterizations of $\mathbb G_n$ can be found in \cite{pal-roy-ijm}.
\begin{thm}[{\cite{costara05}}]\label{char symm poly}
		For $(s_1,\ldots, s_n) \in \mathbb C^n$, the following are equivalent.
		\begin{enumerate}
			\item $(s_1,\ldots, s_n) \in \mathbb G_n$.
			\item $\sup_{|z|\leq 1}|f_s(z)| < 1$, where 
			\[
			f_s(z)= \dfrac{n(-1)^n s_n z^{n-1} + (n-1)(-1)^{n-1} s_{n-1} z^{n-2} +\cdots + (-s_1)}{n - (n-1)s_1 z + \cdots + (-1)^{n-1}s_{n-1} z^{n-1}}.
			\]
			\item $|s_n|<1$ and there exists $(\beta_1 , \ldots, \beta_{n-1}) \in \mathbb G_{n-1}$ such that $s_j = \beta_j + \bar{\beta}_{n-j}s_n$ for $j=1, \ldots, n-1$.
		\end{enumerate}
	\end{thm}
	
Chen, Krantz, and Yuan \cite{Krantz-Chen} have shown that the Friedrichs operator on $\G_2$ is of rank one (see \cite{Krantz-Chen}).
We now show that the same holds for all $\G_n$.
\begin{thm}\label{symm main}
	The Friedrichs operator of $\mathbb G_n$ is of rank one.
\end{thm}
\begin{proof}
The symmetrization map $\pi_n$ is a proper holomorphic covering map with multiplicity $n!$ and Jacobian $J \pi_n = \prod_{1\leq j<k \leq n}^{}(z_j - z_k)$.
Since the Jacobian is a homogeneous polynomial, we are done by Theorem~\ref{main result}.
\end{proof}

A generalization of the symmetrized polydisc called the {\it extended symmetrized polydisc} was introduced by the second author and Pal \cite{pal-roy-ijm}.
These domains are useful in studying the Schwarz lemma for $\mathbb G_n$ (see \cite{pal-roy-caot}) and are related to the $\mu$-synthesis problem as well (see \cite{Roy22}).
The extended symmetrized polydisc $\widetilde{\mathbb G}_n$, $n\ge 2$, is defined as follows.
\begin{align}
	\widetilde{\mathbb G}_n := \Bigg\{ (y_1,\dots,y_{n-1}, q)\in
	\C^n \colon q \in \mathbb D,  y_j = \beta_j + \bar \beta_{n-j} q, \beta_j \in \mathbb C &\text{ and }\\
	|\beta_j|+ |\beta_{n-j}| < {n \choose j} &\text{ for } j=1,\dots, n-1 \Bigg\}.\notag
\end{align}
Note that $\widetilde{\mathbb G}_2 = \mathbb G_2$, and $\mathbb G_n \subsetneq \widetilde{\mathbb G}_n$ for $n\ge 3$.
However, $\widetilde{\mathbb G}_3$ is linearly isomorphic to the tetrablock $\mathbb E$.
Consequently, the Friedrichs operators of $\widetilde{\G}_2$  and $\widetilde{\G}_3$ are of rank one.
It would be interesting to see if the Friedrichs operator continues to have rank one on $\widetilde{\G}_n$ for $n\ge 4$.
	
%%%%%%%%%%%%%%%%%%%%%%%%%%%%%%%%%%%%%%%%%%%%%%%%%%%%%%%%%%%%%%%%%%%%%%%%%%%%

\section*{Acknowledgements}
The authors would like to thank the anonymous referee for pointing out a subtle typo in \eqref{eqn:FSandEucVol}.
The main part of the work reported in this article was done when the second author was a postdoctoral fellow in TIFR CAM.
The second author would like to thank Professor Shyam Sundar Ghoshal for partially funding his position at TIFR CAM from his Inspire faculty-research grant DST/INSPIRE/04/2016/000237.

\def\MR#1{\relax\ifhmode\unskip\spacefactor3000 \space\fi%
  \href{http://www.ams.org/mathscinet-getitem?mr=#1}{MR#1}}

\end{document}